\documentclass[11pt,psamsfonts,twoside]{article}
\usepackage{amssymb}
\usepackage{amssymb}
\usepackage{amssymb}
\usepackage{amssymb}
\usepackage{mathrsfs}
\usepackage{amsmath}
\usepackage{cases}
\usepackage{amsmath}
\usepackage{amsfonts}
\usepackage{amsfonts,latexsym,bm}
\usepackage{amsmath,amsthm}
\usepackage{amssymb,amscd}
\usepackage{amsfonts,amsbsy}
\usepackage{fancyhdr,graphicx}
\usepackage[dvips]{psfrag}
\usepackage{indentfirst}
\numberwithin{equation}{section}

\newtheorem {definition} {Definition} [section]
\newtheorem {theorem} [definition]{Theorem}
\newtheorem {proposition} {Proposition}[section]

\newtheorem {lemma}  [proposition]{Lemma}

 \allowdisplaybreaks

\newcommand{\fg}[1]{\sqrt{#1}}
\newcommand{\fo}{{\vec{f}^{\omega}}}

\begin{document}
\setlength{\parindent}{4ex} \setlength{\parskip}{1ex}
\setlength{\oddsidemargin}{0mm} \setlength{\evensidemargin}{0mm}
\title{{Random-data Cauchy Problem for the Periodic Navier-Stokes Equations
 with Initial Data in Negative-order Sobolev Spaces}\footnote{This work is supported by the
China National Natural Science Fundation under the Grant
no.10771223.}}

\author{Chao Deng$^\dag$, \ \   \ \  \ \ Shangbin Cui\footnote{E-mail:
deng315@yahoo.com.cn; cuisb3@yahoo.com.cn}\\ [0.2cm] {\small
Department of Mathematics, Sun Yat-Sen
University, Guangzhou, }\\
{\small Guangdong 510275, People's Republic of China }}
\date{}

\maketitle
\begin{abstract}
In this paper we study existence of solutions of the
initial-boundary value problems of the Navier-Stokes equations with
a periodic boundary value condition for initial data in the Sobolev
spaces $\mathcal{H}^{s}(\mathbb{T}^N)$ with a negative order
$-1<s<0$, where $N=2, 3$. By using the randomization approach of N.
Burq and N. Tzvetkov, we prove that for almost all
$\omega\in\Omega$, where $\Omega$ is the sample space of a
probability space $(\Omega,\mathcal{A},p)$, for the randomized
initial data
$\vec{f}^\omega\in\mathcal{H}_{\sigma}^{s}(\mathbb{T}^N)$ with
$-1<s<0$, such a problem has a unique local solution.


\end{abstract}
\maketitle


\section{Introduction}
In this paper we study the initial-boundary value problem  of the
Navier-Stokes equations with a periodic boundary value condition:
\begin{equation} \label{eq1.1}
\left\{
\begin{aligned}
  &{\partial_t{\vec{u}}}-\Delta
  \vec{u}+(\vec{u}\cdot\nabla)\vec{u}+\nabla P = 0 \ \ \
  \text{ in }\ \mathbb{R}_+\times \mathbb{R}^N,\\
 & \nabla\cdot\vec{u}=0 \ \ \   \text{in} \  \mathbb{R}_+\times \mathbb{R}^N,\\
 & \vec{u}\;\;\mbox{and}\;\, P \;\, \text{are} \;\,
 \mbox{$1$-periodic in space variables},
 \\
 &\vec{u}(0,x) = \vec{u}_0(x) \ \ \   \text{ for }\  x\in\mathbb{R}^N.
\end{aligned}
\right.
\end{equation}
Here $\vec{u}=\vec{u}(t,x)=\!(u_1(t,x), \cdots,\!u_N(t,x))$ and
$P=P(t,x)$ are unknown vector and scaler functions, respectively,
and $\vec{u}_0=\vec{u}_0(x)=(u_1(x),\cdots, u_N(x))$ is a given
vector function which is $1$-periodic and satisfies the divergence
free condition $\nabla\cdot\vec{u}_0=0$.

Mathematical analysis of the Navier-Stokes equations has a long
history. It goes back to the beginning of the twentieth century. In
\cite{LER34} Leray introduced the concept of weak solutions and
proved existence of global weak solutions associated with
$L^2(\mathbb{R}^N)$ initial data by using an approximation approach
and some weak compactness argument. In \cite{FUJK64} Kato and Fujita
initiated a different approach. They proved well-posedness of the
initial value problem (IVP) of the Navier-Stokes equations in
$\mathcal{H}^s(\mathbb{R}^N)$ for $s\geq\frac{N}{2}-1$, i.e.,
(local) existence and uniqueness of solutions for
$\vec{u}_0\in\mathcal{H}^s(\mathbb{R}^N)$ ($s\geq\frac{N}{2}-1$)
which still belongs to this space for $t>0$ and continuously depends
on the initial data $\vec{u}_0$. This approach was later extended to
various other function spaces by many authors, cf. \cite{CAN04} and
\cite{LEM02} for expositions and references cited therein. In
particular, a famous result due to Cannone and Planchon says that
the IVP of the Navier-Stokes equations is well-posed in the Besov
space $B^s_{p,q}(\mathbb{R}^N)$ for $s>-1$, $ p\geq\frac{N}{1+s}$
and $1\leq q\leq\infty$ (see \cite{CAN95} and \cite{PLA96}). Another
interesting result that must be mentioned is due to Koch and Tataru
(see \cite{KOCT01}). They proved that the IVP of the Navier-Stokes
equations is well-posed in $BMO^{-1}$. This is the largest function
space in which well-posedness of the IVP of the Navier-Stokes
equations has been established. Note that though not rigorously
proved, it is commonly believed that similar results hold true for
the initial-boundary value problem of the Navier-Stokes equations
with a periodic boundary condition.

All the function spaces in which the IVP of the Navier-Stokes are
known to be well-posed are either critical or subcritical spaces. To
explain these concepts let us first consider the scaling invariance
property of these equations. For any $\lambda>0$ we denote
$$
  \vec{u}_\lambda(x,t)=\lambda\vec{u}(\lambda x,\lambda^2 t), \quad
  P_\lambda(x,t)=\lambda P(\lambda^2 x,\lambda^2 t).
$$
It can be easily seen that if $(\vec{u},P)$ is a solution of the
Navier-Stokes equations, then so is $(\vec{u}_\lambda,P_\lambda)$
for any $\lambda>0$. Hence we introduce the scaling
$\varphi_\lambda(x)=\lambda\varphi(\lambda x)$. A function space $X$
on $\mathbb{R}^N$ is respectively called critical, subcritical and
supercritical to the Navier-Stokes equations if its corresponding
homogeneous version space $\dot{X}$ respectively satisfies the
properties $\|\varphi_\lambda\|_{\dot{X}}=\|\varphi\|_{\dot{X}}$,
$\|\varphi_\lambda\|_{\dot{X}}=\lambda^{\mu}\|\varphi\|_{\dot{X}}$
and $\|\varphi_\lambda\|_{\dot{X}}=
\lambda^{-\mu}\|\varphi\|_{\dot{X}}$ for some $\mu>0$ and any
$\lambda>0$ and $\varphi\in\dot{X}$. Though not proved yet, it is
strongly conjectured that in supercritical spaces the IVP of the
Navier-Stokes equations are usually ill-posed. Hence, existence of
solutions for the IVP of the Navier-Stokes equations for initial
data in supercritical spaces are not known except for the
$L^2(\mathbb{R}^N)$ space (which is supercritical for $N\geq 3$ but
critical for $N=2$) and sums of this space with some well-posedness
spaces (cf. \cite{CAL90}, \cite{LEM02} and the recent work of the
second author \cite{CUI}).

Inspired by the interesting works of Burq and Tzvetkov \cite{BT081,
BT082} on the random-data initial-boundary value problems of the
nonlinear wave equations, in this paper we want to construct
solutions for the problem (1.1) for a random class of initial data
in $\mathcal{H}^{s}(\mathbb{T}^N)$ with $-1<s<0$, where
$\mathcal{H}^s(\mathbb{T}^N)$ is the classical $L^2$-type Sobolev
space of index $s$ on the torus $\mathbb{T}^N$, equipped with the
norm $\|\vec{u}\|_{\mathcal{H}^s(\mathbb{T}^N)}=
\|(1-\Delta)^{\frac{s}{2}}\vec{u}\|_{L^2(\mathbb{T}^N)}$. Note that
for any $s<0$, $\mathcal{H}^{s}(\mathbb{T}^N)$ is a supercritical
space, so that it cannot be embedded into any critical or
subcritical spaces (if a space $X$ is embedded into a critical or a
subcritical space $Y$, then $X$ itself must be either a critical or
a subcritical space), and cannot be decomposed into a sum of the
$L^2(\mathbb{R}^N)$ with a critical or a subcritical space either.
Hence existence of solution for initial data in this space can not
be obtained from known results. We must study this problem with new
ideas.

Let us now turn to present the main results of this paper. For this
purpose we first introduce some essential concepts and notations.

First, for a Banach space $X$ of $N$-vector functions or $N$-vector
distributions on $\mathbb{T}^N$, we use the notation $X_\sigma$ to
denote the corresponding closed subspace of $X$ consisting of all
divergence-free elements in $X$, i.e., $\vec{u}\in X_\sigma$ if and
only if $\vec{u}\in X$ and $\nabla\cdot\vec{u}=0$.

Let $(\Omega,\mathcal{A},p)$ be a probability space. We consider a
sequence of independent and identically distributed Gaussian random
variables $(g_n(\omega))_{n\ge0}$, i.e., $g_n\in \mathcal{N}(0,1)$.
Let $(\lambda^2_n)_{n\ge0}$ be the sequence of eigenvalues of
$-\Delta$ on $\mathbb{T}^N$, ordered increasingly and in their
multiplicities and we denote by $(e_n)_{n\ge0}\subseteq
C^{\infty}(\mathbb{T}^N)$  the corresponding sequence of
eigenfunctions, normalized so that they form an orthonormal basis of
$L^2(\mathbb{T}^N)$.
 For any $f\in \mathcal{H}^s(\mathbb{T}^N)$, we have
\begin{equation}\label{eq1.2}
f(x)=\sum_{n\ge0}\alpha_ne_n(x),\ \ \ \alpha_n\in\mathbb{R},\ \ \
\|f\|^2_{\mathcal{H}^s}\approx \sum_{n\ge0}|\alpha_n|^2(1+|\lambda_n|^2)^{s}\Big.
\end{equation}
We consider the map
\begin{equation}\label{eq1.3}
\omega\longmapsto
 f^{\omega}(x)=\sum_{n\ge0}\alpha_ng_n(\omega)e_n(x)\end{equation}
 from $\Omega$ to $\mathcal{H}^s(\mathbb{T}^N)$ (equipped with the
 Borel sigma algebra). Then $f^\omega$ belongs to $L^2(\Omega;
 \mathcal{H}^s(\mathbb{T}^N))$ (cf. \cite{BT081, BT082}) and the map
\eqref{eq1.3} is measurable. The
$\mathcal{H}^s(\mathbb{T}^N)$-valued random variable $f^{\omega}$ is
called the $randomization$ of $f$. Similarly, we define the
randomization of a vector valued function $\vec{f}=(f_1,\dots, f_N)$
by
\begin{equation}\label{eq1.4}\vec{f}^{\omega}(x)=\sum_{n\ge0}({\alpha}^1_{n},
\cdots, {\alpha}^N_{n})e_n(x)g_n(\omega),\end{equation} where
\begin{equation}\nonumber\label{eq1-5}f_1=\sum_{n\ge0}{\alpha}^1_{n}e_n(x),\
\cdots, f_N=\sum_{n\ge0}{\alpha}^N_{n}e_n(x).\end{equation}

 For $1\le q\leq\infty$, $1\le m<\infty$ and $\delta>0$, we introduce
 the following notation: For any measurable function $g(t,x)$ on
 $[0,T]\times\mathbb{T}^N$,
\begin{equation}\label{def1.1}
  \|g\|_{L^{m}_{\delta;T}L^{q}_x}
 =\Big(\int_0^{T}\|t^{\delta}g(t,\cdot)\|^{m}_{L^{q}(\mathbb{T}^N)}dt\Big)^{\frac{1}{m}}.
\end{equation}

  The main results of this paper are as follows:

\begin{theorem}\label{thm1.1}
 Let $N=2,3$ and $-1+\frac{N}{4}<s<0$. Given $\vec{f}\in\mathcal{H}^{s}_\sigma(\mathbb{T}^N)$,
 let $\vec{f}^{\omega}\in L^2(\Omega;\mathcal{H}^{s}_\sigma(\mathbb{T}^N))$ be the
 randomization of $\vec{f}$ given by \eqref{eq1.4}. Then there is an event $\Sigma
 \subseteq\Omega$ with $p(\Sigma)=1$, such that for any $\omega\in\Sigma$, there exists
 a corresponding $T_{\omega}>0$ such that the problem \eqref{eq1.1} with $u_0=\vec{f}^{\omega}$
 has a unique solution $\vec{u}$ for $0\leq t\leq T_{\omega}$, satisfying
\begin{equation}\label{eq1.6}
 \vec{u}-e^{t\Delta}
 \vec{f}^{\omega}\in C([0,T_{\omega}];\mathcal{H}^{s}_{\sigma}(\mathbb{T}^N))
 \bigcap L^{m}_{\delta;T_{\omega}}L^{4}_{x,\sigma}\bigcap
 L^{\frac{8}{4-N}}_{T_{\omega}}L^{4}_{x,\sigma},
 \end{equation}
where $\frac{8}{4-N}<m\le\frac{16}{4-N}$ and
$\delta=\frac{4-N}{8}-\frac{1}{m}$. More precisely, there exist
constants $c_1,c_2>0$ and for every $0<T\leq 1$ an event
$\Omega_{T}$ with the property
\begin{equation}\nonumber
 p(\Omega_{T})\ge
 1-c_1e^{-c_2\|f\|_{\mathcal{H}^s}^{-2}T^{-2\varrho}},
\end{equation}
where $\varrho= \min\{\frac{1}{m},\frac{s}{2}+\!\frac{4-\!N}{8}\}$,
such that for every $\omega\in\Omega_{T}$, the problem \eqref{eq1.1}
with $u_0=\vec{f}^{\omega}$ has a unique solution for $0\leq t\leq
T$ satisfying a similar property as \eqref{eq1.6} $($with
$T_{\omega}$ replaced by $T)$.
\end{theorem}
\begin{theorem}\label{thm1.2}
 Let $N=2,3$ and $-1<s\leq -1+\frac{N}{4}$. Given $\vec{f}\in\mathcal{H}^{s}_\sigma(\mathbb{T}^N)$,
 let $\vec{f}^{\omega}\in L^2(\Omega;\mathcal{H}^{s}_\sigma(\mathbb{T}^N))$ be the
 randomization of $\vec{f}$ given by \eqref{eq1.4}. Then there is an event $\Sigma
 \subseteq\Omega$ with $p(\Sigma)=1$, such that for any $\omega\in\Sigma$, there exists
 a corresponding $T_{\omega}>0$ such that the problem \eqref{eq1.1} with $u_0=\vec{f}^{\omega}$
 has a unique solution $\vec{u}$ for $0\leq t\leq T_{\omega}$, satisfying
\begin{equation}\label{eq1.7}
 \vec{u}-e^{t\Delta}
 \vec{f}^{\omega}\in L^{\frac{4}{1-s}}_{T_{\omega}}L^{\frac{2N}{1+s}}_{x,\sigma}.
\end{equation}
More precisely, there exist constants $c_1,c_2>0$ and for every
$0<T\leq 1$ an event $\Omega_{T}'$ with the property
\begin{equation}\nonumber
 p(\Omega_{T}')\ge
 1-c_1e^{-c_2{\lambda^2}\|f\|_{\mathcal{H}^{s}}^{-2}T^{-\frac{1+s}{2}}},
\end{equation}
such that for every $\omega\in\Omega_{T}'$, the problem
\eqref{eq1.1} with $u_0=\vec{f}^{\omega}$ has a unique solution for
$0\leq t\leq T$ satisfying a similar property as \eqref{eq1.7}
$($with $T_{\omega}$ replaced by $T)$.
\end{theorem}

{\bf Remark}.\ \  First, similarly as in \cite{BT081} we can prove
that for any $s\in\mathbb{R}$ and $\epsilon>0$, if
$\vec{f}\in\mathcal{H}_{\sigma}^s(\mathbb{T}^N)$ is such that
$\vec{f}\notin\mathcal{H}_{\sigma}^{s+\epsilon}(\mathbb{T}^N)$ then
for almost all $\omega\in\Omega$, $\vec{f}^{\omega}\in
\mathcal{H}_{\sigma}^{s}(\mathbb{T}^N)$ but $\vec{f}^{\omega}\notin
\mathcal{H}_{\sigma}^{s+\epsilon}(\mathbb{T}^N)$. Hence the
randomization has no regularizing effect. Next, though the solution
ensured by Theorems \ref{thm1.1} and \ref{thm1.2} are local
solutions, the initial data that we consider can be large, so that
these results are not small data results. Thirdly, note that the
solution satisfies the relation
$\nabla\cdot(\vec{u}-e^{t\Delta}\vec{f}^{\omega})=0$. We shall not
repeat this fact later on. Finally, though our discussion is only
made for the cases $N=2, 3$, from the proof of Theorem \ref{thm1.2}
we easily see that this result (Theorem \ref{thm1.2}) can be
straightforwardly extended to all the cases $N\ge4$ and $-1<s\le 0$.

Later on we shall use $C$ to denote constant which depends only on $N$ and may change
from line to line. We also use $c$ to denote constant which depends on
$N$ and might depends on $s$ and may change from line to line.  We use
the notations $A\sim B $ and $  A\lesssim
B$ stand for $C^{-1}B\le A \le CB$ and $A \le CB$, respectively. The
notations $L^r_{\omega}$, $L^ p_ T$ and $L^q_x$ stand for
$L^r(\Omega)$, $L^p(0,T )$ and $L^q(\mathbb{T}^N)$, respectively,
whereas  we denote $L^{q}_{\sigma}= L_{\sigma}^{q}(\mathbb{T}^N)$
and $L^ p_ T L_{\sigma}^{q} = L^p(0,T
;L_{\sigma}^{q}(\mathbb{T}^N))$.

The rest part is organized as follows. In Section 2 we make some
preliminary discussions.  In Section 3 we give the proofs of
Theorems \ref{thm1.1} and \ref{thm1.2}.

\section{Preliminaries}
In this section we make some preliminary discussion. Let
$(\lambda_n)_{n\ge0}$ be the increasing sequence of all eigenvalues
of the minus Laplace $-\Delta$ on the torus $\mathbb{T}^N$, where
multiple eigenvalues are counted in their multiplicities. Let
$(e_n)_{n\ge0}$ be the corresponding sequence of eigenfunctions
suitably chosen such that they form a normalized orthogonal basis of
$L^2(\mathbb{T}^N)$. Recall that
in the case $N=1$ we have
$$
  \lambda_0=0,\quad \lambda_{2k-1}=\lambda_{2k}=4k^2\pi^2,\quad
  k=1,2,\cdots,
$$
and
$$
  e_0(x)=1,\quad e_{2k-1}(x)=\sqrt{2}\cos(2k\pi x),\quad
  e_{2k}(x)=\sqrt{2}\sin(2k\pi x),\quad  k=1,2,\cdots,
$$
so that
$$
  \|e_n\|_{L^2(\mathbb{T}^1)}=1,\quad
  \|e_n\|_{L^\infty(\mathbb{T}^1)}\leq\sqrt{2},\quad
  n=0,1,2,\cdots.
$$
Since $L^2(\mathbb{T}^N)$ is the $N$-tensor product of
$L^2(\mathbb{T}^1)$, the above inequalities still hold true for
general $N$ when $\sqrt{2}$ is replaced by $(\sqrt{2})^N$. Hence, we
have the following fundamental estimate:

\begin{lemma}\label{lem2.1}
For any $q\in[2,\infty]$ we have
\begin{equation}\label{eq2.1}
  \|e_n\|_{L^q(\mathbb{T}^N)}\lesssim 1.
\end{equation}
\end{lemma}

Next we quote the following preliminary result from \cite{BT081}
(see Lemma 3.1 there):

\begin{lemma}\label{lem2.2}
Let $(g_n(\omega))_{n\ge0}$ be a sequence of Gaussian random
variables. For any $r\ge2$ and $(c_n)_{n\ge0}\in l^2$, we have
\begin{equation}\label{eq2.2}
   \|\sum_{n\ge0}c_ng_n(\omega)\|_{L^r(\Omega)}\lesssim
  \sqrt{r}(\sum_{n\ge0} |c_n|^2)^{1/2}.
\end{equation}
\end{lemma}

In what follows we derive some new estimates. Recall that $N=2,3$.

\begin{lemma}\label{lem2.3}
  Let $-1<s<0$ and $0<T\le 1$. For $f\in \mathcal{H}^{s}(\mathbb{T}^N)$, let
  $f^{\omega}$ be the randomization of $f$ given by \eqref{eq1.3}. Then we have
  the following assertions:

 \medskip
 \noindent $(i)$ If $-1<s\le -1+\frac{N}{4}$,  then for any $r\ge\frac{2N}{1+s}$ we have
\begin{equation}\label{eq2.3}
  \|e^{t\Delta}f^{\omega}\|_{L^{r}_{\omega}L^{\frac{4}{1-s}}_TL^{\frac{2N}{1+s}}_x}\lesssim \fg{r}T^{\frac{1+s}{4}}
   \|f\|_{\mathcal{H}^{s}}.
 \end{equation}
 As a consequence, if we set
 $ E_{\lambda,f,T;1}=\{\omega\in\Omega:
 \|e^{t\Delta}f^{\omega}\|_{L^{\frac{4}{1-s}}_TL^{\frac{2N}{1+s}}_x}\ge
 \lambda\}$,
   then there exist $c_1, c_2>0$ such that for all  $\lambda>0$ and $f\in \mathcal{H}^{s}(\mathbb{T}^N)$,
\begin{equation}\label{eq2.4}
  p(E_{\lambda,f,T;1})\le c_1
   e^{-c_2{\lambda^2}\|f\|_{\mathcal{H}^{s}}^{-2}T^{-\frac{1+s}{2}}}.
\end{equation}

\medskip
 \noindent $(ii)$ If $-1+\frac{N}{4}<s<0$, then for any $r\ge\frac{8}{4-N}$ we have
 \begin{equation}\label{eq2.5}
  \|e^{t\Delta}f^{\omega}\|_{L^{r}_{\omega}L^{\frac{8}{4-N}}_TL^{4}_x}
  \lesssim \fg{r}T^{\frac{s}{2}+\frac{4-N}{8}}\|f\|_{\mathcal{H}^{s}}.
 \end{equation}
 As a consequence, if we set $
 E_{\lambda,f,T;2}=\{\omega\in\Omega:
 \|e^{t\Delta}f^{\omega}\|_{L^{\frac{8}{4-N}}_TL^{4}_x}\ge
 \lambda\}$,
   then there exist $c_1,c_2>0$ such that for all  $\lambda>0$ and $f\in \mathcal{H}^{s}(\mathbb{T}^N)$,
\begin{equation}\label{eq2.51}
  p(E_{\lambda,f,T;2})\le c_1e^{-c_2\lambda^2\|f\|_{\mathcal{H}^{s}}^{2-2}T^{-s+\frac{N-4}{4}}}.
\end{equation}
\end{lemma}

 \begin{proof}\ \
  Given $f\in \mathcal{H}^{s}(\mathbb{T}^N)$, let
   $f=\sum_{n\ge0}\alpha_ne_n$. Then by \eqref{eq1.4} we have
\begin{align*}
  e^{t\Delta}f^{\omega}=\sum_{n\ge0}e^{-t\lambda^2_n}\alpha_ng_n(\omega)e_n(x).
 \end{align*}
We first assume that $-1<s\le -1+\frac{N}{4}$. In this case we have
$\frac{2N}{1+s}\ge\frac{4}{1-s}>2$. Thus for any
$r\ge\frac{2N}{1+s}$, by using the Minkowski inequality
  and Lemmas \ref{lem2.1} $\sim$ \ref{lem2.2} we see that
 \begin{align}
 \nonumber\|e^{t\Delta}f^{\omega}\|^2_{L^{r}_{\omega}L^{\frac{4}{1-s}}_TL^{\frac{2N}{1+s}}_x}&
 \lesssim \|e^{t\Delta}f^{\omega}\|^2_{L^{\frac{4}{1-s}}_TL^{\frac{2N}{1+s}}_x L^{r}_{\omega}}
 \\&\lesssim r \|\sum_{n\ge0}e^{-2t\lambda_n^2}|\alpha_n|^2|e_n(x)|^2\|_{L^{\frac{2}{1-s}}_TL^{\frac{N}{1+s}}_x}
  \nonumber\\& \lesssim r \|\sum_{n\ge0}|\alpha_n|^2e^{-2t\lambda_n^2}\|e_n\|_{L^{\frac{2N}{1+s}}}^2\|_{L^{\frac{2}{1-s}}_T}
 \nonumber\\& \lesssim r \|\sum_{n\ge0}|\alpha_n|^2e^{-2t\lambda_n^2}\|_{L^{\frac{2}{1-s}}_T}
 \nonumber\\& \lesssim r \Big(|\alpha_0|^2T^{\frac{1-s}{2}}+\sum_{n\ge1}|\alpha_n|^2\|e^{-2t\lambda_n^2}\|_{L^{\frac{1}{|s|}}_T}
     T^{\frac{1+s}{2}}\Big)
 \nonumber\\& \lesssim r \Big(|\alpha_0|^2T^{\frac{1-s}{2}}+\sum_{n\ge1}|\alpha_n|^2\lambda_n^{2s}
     T^{\frac{1+s}{2}}\Big)
 \nonumber\\& \lesssim r T^{\min\{\frac{1-s}{2}, \frac{1+s}{2}\}}\sum_{n\ge0} |\alpha_n|^2(1+\lambda_n^2)^{s}
 \nonumber\\& \lesssim r T^{ \frac{1+s}{2}}\|f\|_{\mathcal{H}^s}^2  \label{eq2.7}.
\end{align}
Hence \eqref{eq2.3} follows. From \eqref{eq2.3} and the
Bienaym\'{e}--Tchebychev inequality, it follows that there exists
constant $C>0$ (independent of $f$, $T$ and $r$) such that
\begin{align}\label{eq2.8}
  p(\omega\in\Omega:
\|e^{t\Delta}f^{\omega}\|_{L^{\frac{4}{1-s}}_TL^{\frac{2N}{1+s}}_x}\ge
 \lambda)\le (C\fg{r}\lambda^{-1}\|f\|_{\mathcal{H}^{s}}T^{\frac{1+s}{4}})^{r}.
\end{align}
From this estimate we can easily obtain \eqref{eq2.4}. Indeed, if
$\lambda>0$ is so small such that
$$\Big(\frac{\lambda}{Ce\|f\|_{\mathcal{H}^s}T^{\frac{1+s}{4}}}\Big)^2\le{\frac{2N}{1+s}},$$
then \eqref{eq2.4} follows from the facts that $p(E)\leq 1$ for any
event $E$ and that the function $e^{-x^2}$ has positive lower bound
for $x$ in any bounded set, otherwise we choose
$$r:=\Big(\frac{\lambda}{Ce\|f\|_{\mathcal{H}^{s}}T^{\frac{1+s}{4}}}\Big)^2>\frac{2N}{1+s}.$$
Then from \eqref{eq2.8} we see that
\begin{align*}
  p(E_{\lambda,f,T;1})&
  \le
  e^{-(Ce)^{-2}\lambda^2\|f\|^{-2}_{\mathcal{H}^{s}}T^{-\frac{1+s}{2}}},
\end{align*}
by which \eqref{eq2.4} follows. This proves the assertion $(i)$.

\medskip
Next we assume that $-1+\frac{N}{4}<s<0$. In this case, again by the
Minkowski inequality and Lemmas \ref{lem2.1} $\sim$ \ref{lem2.2} we
see that for any $r\geq\frac{8}{4-N}\geq 4$,
 \begin{align*}
  \|e^{t\Delta}f^{\omega}\|^2_{L^{r}_{\omega}L^{\frac{8}{4-N}}_TL^{4}_x}&
   \lesssim \|e^{t\Delta}f^{\omega}\|^2_{L^{\frac{8}{4-N}}_TL^{4}_xL^{r}_{\omega}}
  \\& \lesssim r\|\sum_{n\ge0}e^{-2t\lambda_n^2}|\alpha_n|^2|e_n(x)|^2\|_{L^{\frac{4}{4-N}}_TL^{2}_x}
  \\& \lesssim  r \|\sum_{n\ge0}e^{-2t\lambda_n^2}|\alpha_n|^2\|e_n\|_{L^{4}_x}^2\|_{L^{\frac{4}{4-N}}_T}
  \\& \lesssim   r\|\sum_{n\ge0}e^{-2t\lambda_n^2}|\alpha_n|^2\|_{L^{\frac{4}{4-N}}_T}
  \\& \lesssim r\Big(|\alpha_0|^2T^{\frac{4-N}{4}}+\sum_{n\ge1}|\alpha_n|^2
   \|e^{-2t\lambda_n^2}\|_{L^{\frac{1}{|s|}}_T}T^{s+\frac{4-N}{4}}\Big)
  \\& \lesssim rT^{\min\{\frac{4-N}{4},s+\frac{4-N}{4}\}}\Big(|\alpha_0|^2+
   \sum_{n\ge1}|\alpha_n|^2\lambda_n^{2s}\Big)
  \\& \lesssim  r T^{s+\frac{4-N}{4}}\|f\|_{\mathcal{H}^{s}}^2 .
\end{align*}
Hence \eqref{eq2.5} follows. From \eqref{eq2.5} and a similar
argument as before we obtain \eqref{eq2.51}.  This proves the
assertion $(ii)$ and completes the proof of Lemma \ref{lem2.3}.
\end{proof}

\begin{lemma}\label{lem2.4}
Let $0<T\le1$, $-1+\frac{N}{4}<s<0<\delta\le\frac{4-N}{16}$,
$\frac{8}{4-N}<m\le\frac{16}{4-N}$ and
$\frac{1}{m}+\delta=\frac{4-N}{8}$. For
$f\in\mathcal{H}^{s}(\mathbb{T}^N) $, let $f^{\omega}$ be the
randomization of $f$ given by \eqref{eq1.3}. Then for any $r\ge m $
we have
\begin{equation}\label{eq2.9}
  \|e^{t\Delta}f^{\omega}\|_{L^{r}_{\omega}L^{m}_{\delta;T}L^{4}_x}\lesssim
  \fg{r}T^{\varrho}\|f\|_{\mathcal{H}^{s}},
 \end{equation}
where $\varrho= \min\{\frac{1}{m},\frac{s}{2}+\!\frac{4-\!N}{8}\}$.
Consequently, if we set
${E}_{\lambda,f,T;3}\!=\!\{\omega\!\in\Omega:
\|e^{t\Delta}f^{\omega}\|_{L^{m}_{\delta;T}L^{4}_x} \ge \lambda\}$,
then there exist $c_1, c_2 > 0$ such that for all $\lambda>0$ and
$f\in \mathcal{H}^{s}(\mathbb{T}^N)$,
\begin{equation}\label{eq2.10}
  p( {E}_{\lambda,f,T;3})\le c_1
  e^{-c_2\lambda^2\|f\|_{\mathcal{H}^{s}}^{-2}T^{-2\varrho}}.
\end{equation}
\end{lemma}

\begin{proof}
Similarly as before, by writing $f=\sum_{n\ge0}\alpha_ne_n$, we have
$$
  t^{\delta}e^{t\Delta}f^{\omega}=\sum_{n\ge0}
  t^{\delta}e^{-t\lambda^2_n}\alpha_ng_n(\omega)e_n(x).
$$
Since $r\ge m>\frac{8}{4-N}\geq 4$, by the Minkowski inequality and
Lemmas \ref{lem2.1} $\sim$ \ref{lem2.2} we have
\begin{align}
 \|t^{\delta} e^{t\Delta}f^{\omega}\|^2_{L^{r}_{\omega}L^{m}_TL^{4}_x}&\lesssim
 r\|t^{\delta} e^{t\Delta}f^{\omega}\|^2_{L^{m}_TL^{4}_xL^{r}_{\omega}}
 \nonumber\\&\lesssim r\|\sum_{n\ge0}t^{2\delta}e^{-2t\lambda_n^2}|\alpha_n|^2|e_n(x)|^2\|_{L^{\frac{m}{2}}_TL^{2}_x}
 \nonumber\\&\lesssim r\|\sum_{n\ge0}|\alpha_n|^2t^{2\delta}e^{-2t\lambda_n^2}\|e_n\|^2_{L^{4}_x}\|_{L^{\frac{m}{2}}_T}
 \nonumber\\&\lesssim r\|\sum_{n\ge0}|\alpha_n|^2t^{2\delta}e^{-2t\lambda_n^2} \|_{L^{\frac{m}{2}}_T}
 \nonumber\\&\lesssim
 r\Big(T^{\frac{2}{m}}|\alpha_0|^2+\sum_{n\ge1}|\alpha_n|^2\|t^{2\delta}e^{-2t\lambda_n^2}\|_{L^{\frac{m}{2}}_T}\Big)
\label{eq2.12}.
\end{align}
To estimates $\|t^{2\delta}e^{-2t\lambda_n^2}\|_{L^{\frac{m}{2}}_T}$
(for $n\ge1$), we consider the two different cases
$-1+\frac{N}{4}\le s\le -\frac{2}{m}$ and  $-\frac{2}{m}<s<0$
separately. In the first case we have
 \begin{align}
  \|t^{2\delta}e^{-2t\lambda_n^2}\|_{L_T^{\frac{m}{2}}}&
  \lesssim{T}^{s+\frac{2}{m}+2\delta}\|t^{-s-\frac{2}{m}}e^{-2t\lambda_n^2}\|_{L_T^{\frac{m}{2}}}
  \nonumber\\&\lesssim{T}^{s+\frac{2}{m}+2\delta}\lambda_n^{2s}
  \lesssim{T}^{s+\frac{4-N}{4}}(1+\lambda_n^{2})^{s}.
\label{eq2.13}
\end{align}
In the second case we have
\begin{align}
 \|t^{2\delta}e^{-2t\lambda_n^2}\|_{L^{\frac{m}{2}}_T}
 \lesssim  T^{2\delta+\frac{2}{m}+s}\|e^{-2t\lambda_n^2}\|_{L_T^{\frac{1}{|s|}}}
 \lesssim {T}^{s+\frac{4-N}{4}}(1+\lambda_n^{2})^{s}.
\label{eq2.14}
\end{align}
Combining \eqref{eq2.12}, \eqref{eq2.13} and \eqref{eq2.14}, we see
that
\begin{align*}
  \|e^{t\Delta}f^\omega\|_{L^{r}_\omega
  L^{m}_{\delta;T}L^4_x}\lesssim
  \sqrt{r}T^{\min\{\frac{1}{m},\frac{s}{2}+\frac{4-N}{8}\}}
  \|f\|_{\mathcal{H}^s}.
\end{align*}
This proves \eqref{eq2.9}. Having proved \eqref{eq2.9},
\eqref{eq2.10} follows from a similar argument as before. This
completes the proof of Lemma \ref{lem2.4}.
\end{proof}

\section{Proofs of Theorems \ref{thm1.1} and \ref{thm1.2}}

In this section we give the proofs of Theorems \ref{thm1.1} and
\ref{thm1.2}. To this end, we convert the problem \eqref{eq1.1} with
$u_0=\fo$ into the following equivalent integral equation:
\begin{equation}\label{eq3.1}
  \vec{u}(t,\cdot)=\vec{u}^{\omega}_{\vec{f}}-\int^t_0
  e^{(t-\tau)\Delta}\mathbb{P}\nabla\cdot(\vec{u}\otimes\vec{u})(\tau,\cdot) d\tau,
\end{equation}
where $\vec{u}^{\omega}_{\vec{f}}=e^{t\Delta}\fo$, $\otimes$ denotes
tensor product between vectors, and $\mathbb{P}$ is the
Helmholtz-Weyl projection operator,  i.e.,
$\mathbb{P}=I+\nabla(-\Delta)^{-1}\nabla$. We make the unknown
variable transformation $\vec{u}\mapsto\vec{v}$ by letting
$\vec{u}=\vec{u}^{\omega}_{\vec{f}}+\vec{v}$. Then the above
integral equation is transformed into the following equivalent one:
\begin{equation}\label{eq3.1}
  \vec{v}(t, \cdot)=-\int^t_0 e^{(t-\tau)\Delta}\mathbb{P}\nabla
  \cdot[(\vec{u}^{\omega}_{\vec{f}}+\vec{v})\otimes
  (\vec{u}^{\omega}_{\vec{f}}+\vec{v})](\tau,\cdot)d\tau.
\end{equation}
We shall use this equivalent form of the integral equation
\eqref{eq3.1} to prove Theorems \ref{thm1.1} and \ref{thm1.2}. In
what follows, we denote by $K^{\omega}_{\vec{f}}$ the following
operator:
\begin{equation}\label{eq3.2}
  K^{\omega}_{\vec{f}}: {\vec{v}}\mapsto
  -\int^t_0 e^{(t-\tau)\Delta}\mathbb{P}\nabla\cdot[
  (\vec{u}^{\omega}_{\vec{f}}+\vec{v})\otimes(\vec{u}^{\omega}_{\vec{f}}
  +\vec{v})](\tau,\cdot)d\tau.
\end{equation}
\medskip

\textbf{Proof of Theorem \ref{thm1.1}}:\ \ Let $m$ and $\delta$ be
as in Theorem \ref{thm1.1}, i.e., $\frac{16}{4-N}\ge
m>\frac{8}{4-N}$ and $\delta=\frac{4-N}{8}-\frac{1}{m}$. Note that
$0<\delta\le \frac{4-N}{16}$. Given $T>0$, let $X=X_T$ be the space
\begin{equation*}
 X=C\Big([0,T];\mathcal{H}^{\frac{N}{2}-1}_\sigma(\mathbb{T}^N)\Big)\bigcap
 L^{m}_{\delta;T}L^{4}_{x,\sigma}\bigcap
 L^{\frac{8}{4-N}}_TL^{4}_{x,\sigma}
\end{equation*}
with norm
$$
  \|u\|_X=\|u\|_{L^{\infty}_T\mathcal{H}^{({N-2})/{2}}_x}+
  \|u\|_{L^{m}_{\delta;T}L^{4}_x}+\|u\|_{L^{{8}/{(4-N)}}_TL^{4}_x}
  \quad \mbox{for}\;\; u\in X.
$$
It is clear that $(X,\|\cdot\|_X)$ is a Banach space. We introduce
the event
\begin{equation}\label{eq3.10'}
  E_{\lambda,\vec{f},T}=\{\omega\in\Omega:
 \|\vec{u}^{\omega}_{\vec{f}}\|_{L^{m}_{\delta;T}L^{4}_{x}}+
 \|\vec{u}^{\omega}_{\vec{f}}\|_{L^{\frac{8}{4-N}}_TL^{4}_{x}}
 \ge\lambda\}.
\end{equation}
In what follows we prove that for any $\omega\in\Omega\backslash
E_{\lambda,\vec{f},T}$, the operator $K^{\omega}_{\vec{f}}$ is
well-defined in $X$, maps $X$ into itself, and is a contraction
mapping when restricted to some closed ball in $X$ provided
$\lambda$ is chosen sufficiently small. To this end, for a given
$\vec{v}\in X$, we estimate the three norms
$\|K^{\omega}_{\vec{f}}(\vec{v})\|_{L^{\infty}_T\mathcal{H}^{({N-2})/{2}}_x}$,
$\|K^{\omega}_{\vec{f}}(\vec{v})\|_{L^{m}_{\delta;T}L^{4}_x}$ and
$\|K^{\omega}_{\vec{f}}(\vec{v})\|_{L^{{8}/{(4-N)}}_TL^{4}_x}$ one
by one.

\textit{Step 1}:\ \ Estimate of
$\|K^{\omega}_{\vec{f}}(\vec{v})\|_{L^{\infty}_T\mathcal{H}^{{(N-2)}/{2}}_x}$.
To estimate this norm, we first prove that for any
$\vec{u},\vec{v}\in L^{m}_{\delta;T}L^{4}_{x,\sigma}$,
\begin{align}\label{eq3.3}
   \|\int_0^te^{(t-\tau)\Delta}\mathbb{P}\nabla\cdot&(\vec{u}\otimes
   \vec{v})(\tau)d\tau\|_{\mathcal{H}^{\frac{N-2}{2}}}
   \lesssim
   \|\vec{u}\|_{L^{m}_{\delta;T}L^{4}_x}
   \|\vec{v}\|_{L^{m}_{\delta;T}L^{4}_x}.
\end{align}
In fact, we have
\begin{align*}
   &\|\!\!\int_0^t\!e^{(t-\tau)\Delta}\mathbb{P}\nabla\cdot
   (\vec{v}\otimes\vec{v})(\tau)d\tau\|_{\mathcal{H}^{\frac{N-2}{2}}}
   \lesssim\int_0^t(t\!-\!\tau)^{-\frac{1}{2}}
   \Big(1\!+\!(t\!-\!\tau)^{-\frac{N-2}{4}}\Big)
   \||\vec{u}(\tau)||\vec{v}(\tau)|\|_{L^{2}_x}d\tau
 \\&\lesssim\Big\{\int_0^t\Big[(t-\tau)^{-\frac{1}{2}}
  \Big(1+(t-\tau)^{-\frac{N-2}{4}}\Big)\tau^{-2\delta}\Big]^{q}d\tau
  \Big\}^{\frac{1}{q}} \|\vec{u}\|_{L^{m}_{\delta;T}L^{4}_x}
  \|\vec{v}\|_{L^{m}_{\delta;T}L^{4}_x}
 \\
 &\lesssim \|\vec{u}\|_{L^{m}_{\delta;T}L^{4}_x}
   \|\vec{v}\|_{L^{m}_{\delta;T}L^{4}_x},
\end{align*}
where $\frac{1}{q}=1-\frac{2}{m}$ (note that $q>1$ and
$(\frac{1}{2}+\frac{N-2}{4}+2\delta)q=1$). This proves
\eqref{eq3.3}. Since $\vec{u}^{\omega}_{\vec{f}}\in
L^{m}_{\delta;T}L^{4}_{x,\sigma}$ and
$\|\vec{u}^{\omega}_{\vec{f}}\|_{L^{m}_{\delta;T}L^{4}_x}\le\lambda$
for $\omega\in\Omega\backslash E_{\lambda,\vec{f},T}$, by
\eqref{eq3.3} we immediately obtain
\begin{align}\label{eq3.4'}
   \|K^{\omega}_{\vec{f}}(\vec{v})\|_{L^{\infty}_T\mathcal{H}^{{(N-2)}/{2}}_x}
   \le C(\lambda^2+ \|\vec{v}\|_{L^{m}_{\delta;T}L^{4}_x}^2).
\end{align}

\textit{Step 2}:\ \  Estimate of
 $\|K^{\omega}_{\vec{f}}(\vec{v})\|_{L^{m}_{\delta;T}L^{4}_{x}}$.
 We first prove that for any $\vec{u},\vec{v}\in\!
 L^m_{\delta;T}L^4_{x,\sigma}\cap\\
 L^{\frac{8}{4-N}}_TL^4_{x,\sigma}$,
\begin{align}\label{eq3.5'}
 \|\int_0^t\!e^{(t-\tau)\Delta}\mathbb{P}\nabla\cdot(\vec{u}\otimes
 \vec{v})(\tau)d\tau &\|_{L^{m}_{\delta;T}L^{4}_{x}}\lesssim
 \min\{(\|\vec{u}\|_{L^m_{\delta;T}L^4_{x}}
 +\|\vec{u}\|_{L^{\frac{8}{4-N}}_{T}L^4_{x}})\|\vec{v}\|_{L^{\frac{8}{4-N}}_TL^4_x},
 \nonumber
 \\& (\|\vec{v}\|_{L^m_{\delta;T}L^4_{x}}+\|\vec{v}\|_{L^{\frac{8}{4-N}}_{T}L^4_{x}})
 \|\vec{u}\|_{L^{\frac{8}{4-N}}_TL^4_x}\}.
\end{align}
 In fact,  we have
 \begin{align*} \|t^{\delta}\! \int_0^te^{(t-\tau)\Delta}\mathbb{P}\nabla\cdot
 (\vec{u}\otimes\vec{v})(\tau)d\tau\|_{L^{m}_TL^{4}_{x}}
  &\lesssim\|\int_0^tt^{\delta}(t-\tau)^{-\frac{1}{2}-\frac{N}{8}}\||\vec{u}(\tau)|
  |{\vec{v}(\tau)}|\|_{L^{2}_{x}}d\tau\|_{L^{m}_T}
  \\&\lesssim\|\int_{\frac{t}{2}}^{t}t^{\delta}(t-\tau)^{-\frac{N+4}{8}}
  \||\vec{u}(\tau)||\vec{v}(\tau)|\|_{L^{2}_{x}}d\tau\|_{L^{m}_T}
  \\&\quad+\|\int_0^{\frac{t}{2}}t^{\delta}(t-\tau)^{-\frac{N+4}{8}}
  \||\vec{u}(\tau)||\vec{v}(\tau)|\|_{L^{2}_{x}}d\tau\|_{L^{m}_T}
  \\&:=I_1+I_2.
\end{align*}
By using the Hardy-Littlewood-Sobolev and the H\"{o}lder
inequalities, we have
$$\begin{aligned}
  I_1&\lesssim\|\int_{\frac{t}{2}}^{t}(t-\tau)^{-\frac{N+4}{8}}
  \|\tau^{\delta}|\vec{u}(\tau)||\vec{v}(\tau)|\|_{L^{2}_{x}}d\tau\|_{L^{m}_T}
 \\&\lesssim\min\{\|\vec{u}\|_{L^{m}_{\delta;T}L^{4}_{x}}\|\vec{v}\|_{L^{\frac{8}{4-N}}_TL^{4}_{x}},
 \|\vec{u}\|_{L^{\frac{8}{4-N}}_TL^{4}_{x}}\|\vec{v}\|_{L^{m}_{\delta;T}L^{4}_{x}}\}.
\end{aligned}$$
Similarly,
$$\begin{aligned}I_2&\lesssim
 \|\int^{\frac{t}{2}}_0(t-\tau)^{-\frac{N+4}{8}+\delta}\||\vec{u}(\tau)||\vec{v}(\tau)|\|_{L^{2}_{x}}d\tau\|_{L^{m}_T}
 \\&\lesssim\|\|\vec{v}(\tau)\|_{L^{4}_{x}}\|\vec{u}(\tau)\|_{L^4_x}\|_{L^{\frac{4}{4-N}}_T}\\&
 \lesssim\|\vec{u}\|_{L^{\frac{8}{4-N}}_TL^{4}_{x}}\|\vec{v}\|_{L^{\frac{8}{4-N}}_TL^{4}_{x}}.
\end{aligned}$$
This proves \eqref{eq3.5'}.  Since $\vec{u}^{\omega}_{\vec{f}}\in
L^{m}_{\delta;T}L^{4}_{x,\sigma}\bigcap
L^{\!\frac{8}{4-N}}_TL^4_{x,\sigma}$ and
$\|\vec{u}^{\omega}_{\vec{f}}\|_{L^{m}_{\delta;T}L^{4}_x}+
\|\vec{u}^{\omega}_{\vec{f}}\|_{L^{\!\frac{8}{4-N}}_TL^{4}_x}\! \le
\lambda$ for $\omega\in\Omega\backslash E_{\lambda,\vec{f},T}$, by
\eqref{eq3.5'} we immediately obtain
\begin{align}\label{eq3.6'}
   \|K^{\omega}_{\vec{f}}(\vec{v})\|_{L^{m}_{\delta;T}L^4_x}
   \le C(\lambda^2+   \|\vec{v}\|_{L^{m}_{\delta;T}L^{4}_x}^2).
\end{align}

 \textit{Step 3}:  Estimate of
 $\|K^{\omega}_{\vec{f}}(\vec{v})\|_{L^{\!\frac{8}{4-N}}_T\!L^{4}_{x}}$.
 Similarly as before, for any $\vec{u},\vec{v}\in
 L^{\!\frac{8}{4-N}}_T\!L^{4}_{x,\sigma}$ we have
\begin{align*}
  \|\int_0^t\!e^{(t-\tau)\Delta}\mathbb{P}\nabla\cdot(\vec{u}\otimes
  \vec{v})(\tau)d\tau\|_{L^{\frac{8}{4-N}}_{T}L^{4}_{x}}
  \lesssim &
  \|\int_0^t\!(t-\tau)^{-\frac{1}{2}-\frac{N}{8}}\||\vec{u}(\tau)|
  |\vec{v}(\tau)|\|_{L^2_x}d\tau\|_{L^{\frac{8}{4-N}}_{T}}
 \\ \lesssim & \|\vec{u}\|_{L^{\frac{8}{4-N}}_{T}L^4_{x}}\|\vec{v}\|_{L^{\frac{8}{4-N}}_TL^4_x}.
\end{align*}
From this inequality and a similar argument as before we conclude
that for any $\omega\in\Omega\backslash E_{\lambda,\vec{f},T}$,
\begin{align}\label{eq3.8'}
   \|K^{\omega}_{\vec{f}}(\vec{v})\|_{L^{\frac{8}{4-N}}_{T}L^4_x}
   \le C(\lambda^2+   \|\vec{v}\|_{L^{\frac{8}{4-N}}_{T}L^{4}_x}^2).
\end{align}

Combining \eqref{eq3.4'}, \eqref{eq3.6'} and \eqref{eq3.8'}, we see
that $K^{\omega}_{\vec{f}}$ is well defined in $X$, and
\begin{equation}\label{eq3.10}
 \|K^{\omega}_{\vec{f}}(\vec{v})\|_{X}\le C_1
   \Big(\lambda^2+\|\vec{v}\|_{X}^2\Big).
\end{equation}
Moreover, by a similar argument we can also prove that for any
$\vec{v}_1,\vec{v}_2\in X$ and $\omega\in \Omega\backslash
E_{\lambda,\vec{f},T}$, there holds
\begin{equation}\label{eq3.11}
 \|K^{\omega}_{\vec{f}}(\vec{v}_1)-K^{\omega}_{\vec{f}}(\vec{v}_2)\|_{X}\le
 C_2\Big(\lambda +\|\vec{v}_1\|_{X} +\|\vec{v}_2\|_{X}\Big)\|\vec{v}_1-\vec{v}_2\|_{X}.
\end{equation}
From \eqref{eq3.10} and \eqref{eq3.11} we easily see that if we set
$\lambda_0=\min\{(2C_1)^{-1},(6C_2)^{-1}\}$, then for any
$0<\lambda\leq\lambda_0$ and $\omega\in\Omega\backslash
E_{\lambda,\vec{f},T}$, the operator $K^{\omega}_{\vec{f}}$ maps the
closed ball $\overline{B}(0,\lambda)$ in $X$ into itself and is a
contraction mapping when restricted to this closed ball, so that it
has a unique fixed point in this closed ball.

Now arbitrarily choose a $\lambda$ in $(0,\lambda_0]$ and fix it. We
denote
\begin{align*}
  \Omega_{T}=\Omega\backslash E_{\lambda,{\vec{f}},T}\!
  =\{\omega\in\Omega:\|\vec{u}^{\omega}_{\vec{f}}\|_{L^{m}_{\delta;T}L^{4}_{\sigma}(\mathbb{T}^N)}
  +\|\vec{u}^\omega_{\vec{f}}\|_{L^{\frac{8}{4-N}}_TL^4_\sigma(\mathbb{T}^N)}< \lambda\}.
\end{align*}
Let $c_1$, $c_2$ and $\varrho$ as in Lemmas \ref{lem2.3} $\sim$
\ref{lem2.4}. Then by noticing that
$E_{\lambda,{\vec{f}},T}\subseteq
E_{\frac{\lambda}{2},{\vec{f}},T;2}\cup
E_{\frac{\lambda}{2},{\vec{f}},T;3}$ and using these lemmas we see
that for any $0<T\leq 1$,
\begin{equation*}
 p(\Omega_{T})=1-p(E_{\lambda,{\vec{f}},T})\ge
 1-2c_1e^{-c_2(\frac{\lambda}{2})^2\|f\|_{\mathcal{H}^s}^{-2}T^{-2\varrho}}.
\end{equation*}
Hence the second part of the assertion in Theorem \ref{thm1.1} is
proved. Next, for any $j\in\mathbb{N}$, we choose
$T_{j}:=T(j,N,s,\lambda,\|\vec{f}\|_{\mathcal{H}^s})\in(0, 1)$ so
small such that
$$
  c_1e^{-c_2(\frac{\lambda}{2})^2\|\vec f\|_{\mathcal{H}^s}^{-2}T_j^{-2\varrho}}\leq
  2^{-j-1}.
$$
Then
\begin{equation}\nonumber
 p(\Omega_{T_{j}})\ge
 1-2c_1e^{-c_2(\frac{\lambda}{2})^2\|f\|_{\mathcal{H}^s}^{-2}T_j^{-2\varrho}}\ge
 1-2^{-j}.
\end{equation}
We now let $\Sigma=\bigcup_{j\in\mathbb{N}}\Omega_{T_j}$. Clearly,
$p(\Sigma)=1$. Since for any $\omega\in\Sigma$ there exists a
corresponding $j=j(\omega)\in\mathbb{N}$ such that $\omega\in
\Omega_{T_{j(\omega)}}$, by letting $T_{\omega}=T_{j(\omega)}$, we
see that the first part of the assertion in Theorem \ref{thm1.1}
follows. This completes the proof of Theorem \ref{thm1.1}.

\medskip
\textbf{Proof of Theorem \ref{thm1.2}}:\ \ Let $s$ be as in Theorem
\ref{thm1.2}, i.e., $-1<s\le -1+\frac{N}{4}$. For given $T>0$, let
$Y$ be the Banach space
$L^{\frac{4}{1-s}}_TL^{\frac{2N}{1+s}}_{x,\sigma}$. By using some
similar inequalities as before we see that for any $v\in Y$,
\begin{align}\label{eq3.12}
   \|\int_0^t\!\!e^{(t-\tau)\Delta}\mathbb{P}
   \nabla\!\cdot\!(\vec{u}\otimes\vec{v})(\tau)d\tau\|_{Y}
   &=\|\int_0^t\!\!e^{(t-\tau)\Delta}\mathbb{P}
   \nabla\!\cdot\!(\vec{u}\otimes\vec{v})(\tau)d\tau\|_{L^{\!\frac{4}{1-s}}_T\!L^{\!\frac{2N}{1+s}}_x}
\nonumber
   \\&\lesssim \|\!\int_0^t\! (t\!-\!\tau)^{-\frac{1}{2}-\frac{1+s}{4}}
   \||\vec{u}(\tau)||\vec{v}(\tau)|\|_{\!L^{\!\frac{N}{1+s}}_{x}}d\tau
   \|_{L^{\!\frac{4}{1-s}}_T}
\nonumber
 \\&\lesssim\||\vec{u}||\vec{v}|\|_{L^{\frac{2}{1-s}}_TL^{\!\frac{N}{1+s}}_{x}}
 \lesssim \|\vec{u}\|_{Y}\|\vec{v}\|_{Y}.
\end{align}
Now set
\begin{equation*}
  \widetilde{E}_{\lambda,\vec{f},T}=\{\omega\in\Omega:
 \|\vec{u}^{\omega}_{\vec{f}}\|_{Y}=
 \|\vec{u}^{\omega}_{\vec{f}}\|_{L^{\frac{4}{1-s}}_TL^{\frac{2N}{1+s}}_{x}}
 \ge\lambda\}.
\end{equation*}
Since $\|\vec{u}^\omega_{\vec{f}}\|_{Y}\le \lambda$ for $\omega\in
\Omega\backslash\widetilde E_{\lambda,\vec{f},T}$, using
\eqref{eq3.12} we easily see that for any $\omega\in
\Omega\backslash\widetilde E_{\lambda,\vec{f},T}$ and $v,\vec{v}_1,
\vec{v}_2\in Y$ there hold
\begin{align*}
 \|K^\omega_{\vec{f}}(\vec{v})\|_{Y}
 \le & C_1(\lambda^2+\|\vec{v}\|_{Y}),
\\
 \|K^\omega_{\vec{f}}(\vec{v_1})-K^\omega_{\vec{f}}(\vec{v_2})\|_Y\le &
 C_2\Big(\lambda+\|\vec{v}_1\|_Y+\|\vec{v}_2\|_Y\Big)\|\vec{v}_1-\vec{v}_2\|_Y.
\end{align*}
From these estimates and a similar argument as in the proof of
Theorem \ref{thm1.1}, we obtain the desired assertion in Theorem
\ref{thm1.2}. This completes the proof of Theorem \ref{thm1.2}.

\bibliographystyle{amsplain}
\end{document}